\documentclass[11pt]{article}

\usepackage[a4paper,margin=1.15in]{geometry}
\usepackage{amsmath,amssymb,amsthm,mathtools}
\usepackage{enumitem}
\usepackage{array}
\usepackage{booktabs}
\usepackage{microtype}
\usepackage[colorlinks=true,linkcolor=blue,citecolor=blue,urlcolor=blue]{hyperref}

\newtheorem{theorem}{Theorem}[section]
\newtheorem{proposition}[theorem]{Proposition}
\newtheorem{lemma}[theorem]{Lemma}
\newtheorem{corollary}[theorem]{Corollary}
\theoremstyle{definition}

\newtheorem{remark}[theorem]{Remark}
\newtheorem{question}[theorem]{Question}

\newcommand{\Fp}{\mathbb F_p}
\newcommand{\Fpb}{\overline{\mathbb F}_p}
\newcommand{\PP}{\mathbb P^1}
\newcommand{\Aut}{\operatorname{Aut}}

\newcommand{\PGL}{\operatorname{PGL}}

\title{Artin--Schreier Curves of Prescribed $p$-Rank\\Without Extra Automorphisms}

\author{Saeed Tafazolian\thanks{Department of Mathematics, Institute of Mathematics, Statistics and Scientific Computing (IMECC), University of Campinas (UNICAMP), Campinas, SP, Brazil. Email: \texttt{saeed@unicamp.br}.}~and~Ahmad Kazemifard\thanks{Department of Mathematics, Shahid Chamran University of Ahvaz, Ahvaz, Iran. Email: \texttt{a.kazemifard@scu.ac.ir}.}}
\date{}

\begin{document}
\maketitle

\begin{abstract}
Let $p$ be an odd prime. We study the Artin--Schreier analogue of Poonen's problem of constructing curves without extra automorphisms. An Artin--Schreier curve carries an unavoidable deck group $C_p$, so the smallest possible full geometric automorphism group inside the Artin--Schreier locus is $C_p$. We determine the exact numerical range in which this minimum is attained, for every $p$-rank. More precisely, for integers $r\ge0$ and $N\ge1$, there exists an Artin--Schreier curve $X/\Fp$ with
\[
 g(X)=\frac{p-1}{2}N,\qquad \gamma(X)=r(p-1),\qquad
 \Aut_{\Fpb}(X)\cong C_p
\]
if and only if either $r=0$, $p\nmid N+1$, $N\ge3$ and $(p,N)\ne(3,3)$, or $r\ge1$, $N\ge2r$ and $(r,N)\ne(1,2)$. All models are defined over the prime field by explicit equations together with elementary choices of auxiliary polynomials, and every excluded admissible case is shown to carry an extra geometric automorphism. The main structural input is the theorem of Valentini and Madan on the uniqueness of the rational subfield of index $p$ outside three exceptional families. We exclude these families by intrinsic arguments: the genus and $p$-rank in most cases, Nakajima's bound on $p$-subgroups of automorphism groups of ordinary curves for the ordinary family, and the Lehr--Matignon normal form for isomorphisms of one-point covers for the monomial family.
\end{abstract}



\section{Introduction}

Poonen proved that for every field $k$ and every $g\ge3$ there exists a curve of genus $g$ over $k$ with trivial automorphism group \cite{PoonenI}. He also proved the hyperelliptic analogue: for every field $k$ and every $g\ge2$ there exists a hyperelliptic curve $X/k$ whose automorphism group consists only of the identity and the hyperelliptic involution \cite{PoonenII}. In the hyperelliptic locus the involution is unavoidable, so the natural condition is the absence of \emph{extra} automorphisms.

The same question has a natural wild analogue. If
\[
 X:\qquad y^p-y=f(x)
\]
is an Artin--Schreier curve in characteristic $p$, then
\[
 (x,y)\longmapsto (x,y+c),\qquad c\in\Fp,
\]
gives a deck group $C_p\le \Aut_{\Fpb}(X)$. Thus the smallest possible geometric automorphism group inside the Artin--Schreier locus is $C_p$. Our main result gives a complete numerical answer inside this locus, including the one-branch-point case.

\begin{theorem}[Sharp numerical classification]\label{thm:main}
Let $p$ be an odd prime, and let $r\ge0$ and $N\ge1$ be integers. There exists a smooth projective geometrically irreducible Artin--Schreier curve $X/\Fp$ such that
\[
 g(X)=\frac{p-1}{2}N,
 \qquad
 \gamma(X)=r(p-1),
 \qquad
 \Aut_{\Fpb}(X)\cong C_p
\]
if and only if one of the following holds:
\begin{enumerate}[label=\textup{(\roman*)}]
\item $r=0$, $p\nmid N+1$, $N\ge3$, and $(p,N)\ne(3,3)$;
\item $r\ge1$, $N\ge2r$, and $(r,N)\ne(1,2)$.
\end{enumerate}
\end{theorem}

Here $\gamma(X)$ denotes the $p$-rank of $X$. Every Artin--Schreier curve has genus of the form $\frac{p-1}2N$ and $p$-rank of the form $r(p-1)$ (Proposition~\ref{prop:formulas}), so Theorem~\ref{thm:main} covers all numerically possible pairs. For positive $p$-rank, the inequality $N\ge2r$ is forced by the number of poles, and the unique missing boundary point $(r,N)=(1,2)$ has an unavoidable involution. For $p$-rank zero there is exactly one branch point, of lower jump $d=N+1$; hence $p\nmid N+1$ is necessary. The only further obstructions are $N=1$, $N=2$, and the characteristic-three point $(p,N)=(3,3)$; in each of them \emph{every} admissible curve has an extra automorphism (Proposition~\ref{prop:zero-obstructions}).

The following consequence is the direct Artin--Schreier analogue of Poonen's ``no extra automorphisms'' theorem.

\begin{corollary}[Poonen-type existence theorem]\label{thm:poonen-type}
Let $p$ be an odd prime. For every integer $N\ge3$ there exists a smooth projective geometrically irreducible Artin--Schreier curve $X/\Fp$ such that
\[
 g(X)=\frac{p-1}{2}N,\qquad
 \gamma(X)=p-1,\qquad
 \Aut_{\Fpb}(X)\cong C_p.
\]
\end{corollary}

The lower bound $N\ge3$ in this corollary is sharp: when $N=2$, both lower jumps are $1$, and every such cover has an involution in addition to the deck group; see Proposition~\ref{prop:N2-impossible}. For every $r\ge2$, by contrast, the positive-$p$-rank boundary $N=2r$ is attained.

\paragraph{Relation with prescribed $p$-rank results.}
The $p$-rank strata of Artin--Schreier curves were studied by Pries and Zhu \cite{PriesZhu}. If the poles of $f$ have orders $d_1,\ldots,d_s$, all prime to $p$, then
\[
 g(X)=\frac{p-1}{2}\left(-2+\sum_{i=1}^s(d_i+1)\right),
 \qquad
 \gamma(X)=(s-1)(p-1).
\]
Thus prescribing the $p$-rank of an Artin--Schreier curve amounts to prescribing the number of poles, whereas controlling the full geometric automorphism group is a separate condition.

Achter, Glass and Pries proved that, over an algebraically closed field $k$ of characteristic $p$, for every $g\ge3$ and $0\le f\le g$ there exists a curve of genus $g$ and $p$-rank $f$ with trivial automorphism group, together with the corresponding statement for hyperelliptic curves \cite[Corollary~1.1]{AGP}. Their result is moduli-theoretic: they point out that it gives no information about the field of definition, and they record as open whether such curves exist over $\Fp$ in general, and, in odd characteristic, whether the hyperelliptic analogue holds over $\Fp$ \cite[Remark~1.2]{AGP}. Our results do not solve those problems, since an Artin--Schreier curve necessarily has a subgroup $C_p$. Instead, we ask for the minimal possible automorphism group \emph{inside the Artin--Schreier locus}, require the curve to be defined over $\Fp$, and determine the full geometric automorphism group of explicit models.

\paragraph{Relation with the one-branch-point theory.}
Lehr and Matignon study $p$-cyclic covers of $\PP$ ramified at exactly one point, that is, the $p$-rank-zero Artin--Schreier case considered in Section~\ref{sec:one-pole}. Their main concern is the structure of the $p$-Sylow subgroups of the full automorphism group: they prove that the wild inertia group at the unique branch point is a $p$-Sylow subgroup and describe it as a central extension of the deck group by an elementary abelian $p$-group \cite[Theorem~1.1]{LehrMatignon}. They also show that an isomorphism between two such covers may be modified by an automorphism so that it descends to an affine map of the base \cite[Proposition~3.3]{LehrMatignon}; we use this to exclude the monomial exceptional family. We do not claim a new structural classification of one-point covers. Our use of the one-point theory is different: we determine, over the prime field, the exact numerical range in which the \emph{full geometric} automorphism group can be as small as the deck group.

\paragraph{Method.}
The structural input is the theorem of Valentini and Madan \cite{VM80}: outside three explicit exceptional Artin--Schreier families, the rational subfield of index $p$ is unique, so every geometric automorphism descends to a fractional linear transformation of the $x$-line. A convenient modern statement of the exceptional families is \cite[Corollary~2.16]{DGLMS}:
\begin{equation}\label{eq:exceptional-families}
 y^p-y=\frac{a}{x^p-x}\quad(a\in k^\times),\qquad
 y^p-y=\frac{1}{x^{\lambda}}\quad(\lambda\mid p+1),\qquad
 y^3-y=\frac{i}{x(x-1)}\quad(i^2=2).
\end{equation}
Once descent is available, the problem becomes a rigidity question for the branch data: the induced projective transformation must preserve the branch points with their lower jumps (Lemma~\ref{lem:jumps}), and must satisfy the coboundary equation~\eqref{eq:lift}. Our models are designed so that this forces the identity. The main technical point is to certify, for each model, that it is not birational to an exceptional curve. Most cases are excluded by the genus and the $p$-rank alone (Lemma~\ref{lem:exceptional-invariants}). The two remaining coincidences are handled intrinsically:
\begin{itemize}
\item at the ordinary boundary $(r,N)=(p-1,2p-2)$, the curves $y^p-y=a/(x^p-x)$ are ordinary, and Nakajima's bound \cite[Theorem~1]{Nakajima} forces all their Artin--Schreier structures to be conjugate (Lemma~\ref{lem:E-rigid});
\item in $p$-rank zero, when $N+1\mid p+1$, the Lehr--Matignon normal form forces a curve birational to $y^p-y=x^{-\lambda}$ to have a very special defining polynomial (Lemma~\ref{lem:monomial}).
\end{itemize}
These arguments are independent of the standard-form machinery of \cite{DGLMS}, whose Lemmas~3.6 and~3.8 give alternative proofs of the same exclusions.

\paragraph{Organization.}
Section~\ref{sec:prelim} records the genus and $p$-rank formulas, the Valentini--Madan descent theorem, and the exclusion of the exceptional families. Section~\ref{sec:one-pole} proves the sharp one-branch-point classification, including the small obstructions. Section~\ref{sec:two-poles} treats two branch points. Section~\ref{sec:many-poles} treats three or more branch points and proves the full sharp range $N\ge2r$. Section~\ref{sec:proofs} assembles the proofs and records consequences for finite fields and Jacobians; Section~\ref{sec:remarks} summarizes the models and lists further questions.

\section{Artin--Schreier covers and descent of automorphisms}\label{sec:prelim}

Let $k=\Fpb$. By an \emph{Artin--Schreier curve} over a field $F\subseteq k$ of characteristic $p$ we mean the smooth projective model $X$ of an equation
\[
 X:\qquad y^p-y=f(x),\qquad f(x)\in F(x)\setminus\{h^p-h:\ h\in k(x)\}.
\]
Replacing $f$ by $f+h^p-h$ does not change the extension $k(X)/k(x)$. We choose an \emph{Artin--Schreier reduced} representative: at every pole, its principal part contains no term whose pole order is divisible by $p$. In particular, every pole order of $f$ is prime to $p$; when $f\in k[x]$, this means that no nonconstant monomial of $f$ has exponent divisible by $p$. A nonconstant reduced rational function with a pole cannot be of the form $h^p-h$: indeed, every pole of $h^p-h$ has order divisible by $p$. Thus the reduced equations used below define nontrivial degree-$p$ Artin--Schreier extensions, and their smooth projective models are geometrically irreducible.

\begin{proposition}[Riemann--Hurwitz and Deuring--Shafarevich]\label{prop:formulas}
Suppose that $f$ has exactly $s$ poles $P_1,\ldots,P_s$ on $\PP$, of orders $d_1,\ldots,d_s$, where $p\nmid d_i$ for every $i$. Then
\[
 g(X)=\frac{p-1}{2}\left(-2+\sum_{i=1}^s(d_i+1)\right)
 \qquad\text{and}\qquad
 \gamma(X)=(s-1)(p-1).
\]
\end{proposition}

\begin{proof}
The different exponent above a pole of order $d_i$ is $(p-1)(d_i+1)$ \cite[Proposition~3.7.8]{Stichtenoth}, so the genus formula follows from Riemann--Hurwitz. The $p$-rank formula is the Deuring--Shafarevich formula; see \cite{Subrao} and \cite[Lemma~2.6]{PriesZhu}.
\end{proof}

Throughout, we write
\[
 N:=-2+\sum_{i=1}^s(d_i+1),\qquad r:=s-1,
\]
so that $g(X)=\frac{p-1}2N$ and $\gamma(X)=r(p-1)$. The deck group
\[
 T=\{\tau_c:(x,y)\mapsto(x,y+c):c\in\Fp\}
\]
is cyclic of order $p$.

\paragraph{Convention.} An automorphism $\sigma$ of $X$ is identified with the corresponding automorphism of $k(X)$. If $\sigma$ preserves $k(x)$, then $\sigma(x)=M(x)$ for some $M\in\PGL_2(k)$. The induced map on the points of $\PP$ is $M$ or $M^{-1}$, depending on whether one lets automorphisms act by pull-back or by push-forward. All our arguments only use that this point map permutes the branch points of $X\to\PP$ and preserves their lower jumps, and these properties are insensitive to inversion.

\begin{proposition}[Valentini--Madan]\label{prop:VM}
Let $X/k$ be an Artin--Schreier curve. Assume that $X$ is not birational over $k$ to a curve in one of the three families \eqref{eq:exceptional-families}. Then the rational subfield $k(x)\subset k(X)$ of index $p$ is unique. Consequently every $\sigma\in\Aut_k(X)$ preserves $k(x)$ and induces an element $M_\sigma\in\PGL_2(k)$ on the $x$-line.
\end{proposition}

\begin{proof}
The uniqueness statement is \cite[Theorem~6]{VM80}; see also \cite[Corollary~2.16]{DGLMS}. The final assertion follows because $\sigma(k(x))$ is again a rational subfield of index $p$.
\end{proof}

\begin{lemma}[Artin--Schreier generators]\label{lem:AS-generators}
Let $K=k(x,y)$ with $y^p-y=f\in k(x)$ and $[K:k(x)]=p$. If $z\in K$ satisfies $K=k(x,z)$ and $z^p-z\in k(x)$, then $z=uy+h$ with $u\in\Fp^\times$ and $h\in k(x)$.
\end{lemma}

\begin{proof}
Let $\tau$ be the generator $y\mapsto y+1$ of $\operatorname{Gal}(K/k(x))$. Since $(\tau z-z)^p-(\tau z-z)=\tau(z^p-z)-(z^p-z)=0$, we have $\tau z-z=u\in\Fp$, and $u\ne0$ because $z\notin k(x)$. Then $\tau(z-uy)=z-uy$, so $h:=z-uy\in k(x)$.
\end{proof}

\begin{lemma}[Form of an automorphism]\label{lem:lift}
Under the hypotheses of Proposition \ref{prop:VM}, every automorphism of $X$ is of the form
\[
 (x,y)\longmapsto\left(M(x),u y+h(x)\right),
\]
where $M\in\PGL_2(k)$, $u\in\Fp^\times$, and $h(x)\in k(x)$. These data satisfy
\begin{equation}\label{eq:lift}
 f(M(x))-u f(x)=h(x)^p-h(x).
\end{equation}
If $M=1$, then the automorphism belongs to the deck group $T$. Conversely, for any Artin--Schreier curve $X$ \textup{(}exceptional or not\textup{)}, every triple $(M,u,h)$ satisfying \eqref{eq:lift} defines an automorphism $(x,y)\mapsto(M(x),uy+h(x))$ of $X$.
\end{lemma}

\begin{proof}
By Proposition~\ref{prop:VM}, $\sigma(x)=M(x)$ with $M\in\PGL_2(k)$. The element $z=\sigma(y)$ satisfies $k(X)=k(x,z)$ and $z^p-z=\sigma(f(x))=f(M(x))\in k(x)$, so Lemma~\ref{lem:AS-generators} gives $z=uy+h$. Substituting into the Artin--Schreier equation gives \eqref{eq:lift}, using $u^p=u$. If $M=1$, then $\sigma$ fixes $k(x)$ and therefore lies in $\operatorname{Gal}(k(X)/k(x))=T$. The converse is immediate. See also \cite[Lemma~2.15 and Corollary~2.16]{DGLMS}.
\end{proof}

\begin{lemma}[Preservation of lower jumps]\label{lem:jumps}
Let $\sigma\in\Aut_k(X)$ induce $M\in\PGL_2(k)$ as above. Then the induced point map on $\PP$ preserves the branch locus of $X\to\PP$ together with the lower ramification jump attached to each branch point.
\end{lemma}

\begin{proof}
By Proposition \ref{prop:VM}, $\sigma$ preserves the unique rational subfield $k(x)$. Hence it normalizes its Galois group: $\sigma T\sigma^{-1}=T$. Let $Q$ lie above a branch point $P$. Conjugation by $\sigma$ carries the lower ramification filtration of $T$ at $Q$ to the lower ramification filtration of $T$ at $\sigma(Q)$:
\[
 \sigma T_i(Q)\sigma^{-1}=T_i(\sigma(Q))\qquad(i\ge0).
\]
This follows directly from the valuation definition of the lower ramification groups and the identity $v_{\sigma(Q)}(\sigma(z))=v_Q(z)$. Consequently the last nontrivial lower ramification group occurs at the same index. For a reduced Artin--Schreier equation, a pole of order $d$ prime to $p$ has lower jump $d$; see \cite[Chapter~IV]{Serre} or \cite[Proposition~3.7.8]{Stichtenoth}.
\end{proof}

\subsection{Exclusion of the exceptional families}

\begin{lemma}[Invariants of the exceptional families]\label{lem:exceptional-invariants}
In the notation $(r,N)$ above, the curves in \eqref{eq:exceptional-families} have the following branch data:
\begin{enumerate}[label=\textup{(\arabic*)}]
\item $y^p-y=a/(x^p-x)$: $p$ simple poles, so $(r,N)=(p-1,2p-2)$ and $g=\gamma=(p-1)^2$; in particular the curve is ordinary;
\item $y^p-y=x^{-\lambda}$, $\lambda\mid p+1$: one pole of order $\lambda$, so $(r,N)=(0,\lambda-1)$;
\item $y^3-y=i/(x(x-1))$: $p=3$ and two simple poles, so $(r,N)=(1,2)$.
\end{enumerate}
Consequently, if $X$ is an Artin--Schreier curve whose pair $(r,N)$ is not one of $(p-1,2p-2)$, $(0,\lambda-1)$ with $\lambda\mid p+1$, or $(1,2)$, then Proposition~\ref{prop:VM} applies to $X$.
\end{lemma}

\begin{proof}
This is Proposition~\ref{prop:formulas} applied to the displayed equations, together with the fact that $g$ and $\gamma$, hence $(r,N)$, are birational invariants.
\end{proof}

The following elementary fact is used repeatedly.

\begin{lemma}[Reduced polynomials]\label{lem:reduced-unique}
Let $Q\in k[x]$. If $Q^p-Q$ contains no monomial $x^{jp}$ with $j\ge1$, then $Q$ is constant. In particular, two reduced polynomials that differ by an element of $\{Q^p-Q:Q\in k[x]\}$ differ by a constant.
\end{lemma}

\begin{proof}
If $e=\deg Q\ge1$, the coefficient of $x^{pe}$ in $Q^p-Q$ is the $p$-th power of the leading coefficient of $Q$, because $pe>e$.
\end{proof}

\begin{lemma}[Rigidity of the ordinary exceptional family]\label{lem:E-rigid}
Let $a\in k^\times$, let $E_a$ be the curve $y^p-y=a/(x^p-x)$, and let $T_E$ be its deck group.
\begin{enumerate}[label=\textup{(\alph*)}]
\item Every subgroup $P\le\Aut_k(E_a)$ of order $p$ with $E_a/P\cong\PP$ is conjugate to $T_E$ in $\Aut_k(E_a)$.
\item Let $X:y^p-y=f(x)$ be a reduced Artin--Schreier curve birational to $E_a$. Then there are $M\in\PGL_2(k)$, $u\in\Fp^\times$ and a constant $h\in k$ such that
\begin{equation}\label{eq:E-transfer}
 \frac{a}{M(x)^p-M(x)}=u\,f(x)+h.
\end{equation}
In particular, all poles of $f$ are simple, and the branch locus of $X\to\PP$ is projectively equivalent to $\mathbb F_p\subset\PP$; hence it has a nontrivial stabilizer in $\PGL_2(k)$ \textup{(}a conjugate of $x\mapsto x+1$\textup{)}.
\item Suppose $p=3$ and $X$ is as in \textup{(b)}. Order the three branch points of $X\to\PP$, let $z$ be the coordinate sending them to $\infty,0,1$, and write $f=\alpha z+\beta/z+\gamma/(z-1)+\delta$. Then $(\alpha,\beta,\gamma)=(-s,s,s)$ for some $s\in k^\times$.
\end{enumerate}
\end{lemma}

\begin{proof}
(a) The curve $E_a$ is ordinary of genus $(p-1)^2\ge4$ by Lemma~\ref{lem:exceptional-invariants}. By Nakajima's bound \cite[Theorem~1]{Nakajima} (see also \cite[Theorem~11.84]{HKT}), every $p$-subgroup of $\Aut_k(E_a)$ has order at most
\[
 \frac{p}{p-2}\bigl(\gamma(E_a)-1\bigr)=\frac{p}{p-2}\bigl((p-1)^2-1\bigr)=p^2.
\]
The automorphisms $\tau:(x,y)\mapsto(x,y+1)$ and $\rho:(x,y)\mapsto(x+1,y)$ generate a subgroup $S\cong C_p\times C_p$, which is therefore a $p$-Sylow subgroup. Its subgroups of order $p$ are $T_E=\langle\tau\rangle$, $R=\langle\rho\rangle$, and $\langle\rho\tau^{c}\rangle$ with $c\in\Fp^\times$.

The element $\rho\tau^c$ ($c\ne0$) acts without fixed points: it moves every point lying over a finite value of $x$, and over $x=\infty$ (which is not a branch point, $f$ vanishing there) it permutes the $p$ points, on which $y$ takes the values of $\Fp$, by $y\mapsto y\pm c$. By Riemann--Hurwitz the quotient has genus $1+(g(E_a)-1)/p=p-1>0$. The equation of $E_a$ can be written $(x^p-x)(y^p-y)=a$, so $\iota:(x,y)\mapsto(y,x)$ is an automorphism with $\iota R\iota^{-1}=T_E$; in particular $E_a/R\cong\PP$.

Now let $P$ be as in (a). By Sylow's theorem, $g^{-1}Pg\le S$ for some $g\in\Aut_k(E_a)$, and $E_a/g^{-1}Pg\cong E_a/P\cong\PP$. Hence $g^{-1}Pg\in\{T_E,R\}$, and $P$ is conjugate to $T_E$.

(b) Let $\varphi:X\to E_a$ be an isomorphism. The group $\varphi T_X\varphi^{-1}$ has order $p$ and quotient isomorphic to $X/T_X\cong\PP$. By (a), after composing $\varphi$ with an automorphism of $E_a$ we may assume $\varphi T_X\varphi^{-1}=T_E$. Then the induced isomorphism $\varphi^*:k(E_a)\to k(X)$ maps the fixed field $k(x_E)$ of $T_E$ onto the fixed field $k(x)$ of $T_X$, so $\varphi^*(x_E)=M(x)$ for some $M\in\PGL_2(k)$. Moreover $z=\varphi^*(y_E)$ satisfies $k(X)=k(x,z)$ and $z^p-z=a/(M(x)^p-M(x))\in k(x)$. Lemma~\ref{lem:AS-generators} gives $z=uy+h$ with $h\in k(x)$, whence
\[
 \frac{a}{M(x)^p-M(x)}=u\,f(x)+h(x)^p-h(x).
\]
The left side has only simple poles. If $h$ had a pole of order $e\ge1$ at some point, then $h^p-h$ would have a pole of order $pe$ there, which could not be cancelled by $uf$, whose poles have order prime to $p$; the right side would then have a pole of order at least $p\ge3$. Hence $h$ is constant, and \eqref{eq:E-transfer} holds. It follows that the poles of $f$ are exactly the poles of the left side, namely the points mapped into $\mathbb F_p$ by $M$; so the branch locus of $X$ is the image of $\mathbb F_p$ under an element of $\PGL_2(k)$, and $\mathbb F_p$ is preserved by $x\mapsto x+1$.

(c) Let $\Phi\in\PGL_2(k)$ be the coordinate change with $x=\Phi(z)$ sending $\infty,0,1$ to the ordered branch points, so that $f(\Phi(z))=\alpha z+\beta/z+\gamma/(z-1)+\delta$. By (b),
\[
 f(\Phi(z))=u^{-1}f_E(\Lambda(z))+\text{const},\qquad f_E=\frac{a}{x^3-x},\quad \Lambda:=M\circ\Phi .
\]
Comparing poles, $\Lambda$ maps $\{\infty,0,1\}$ bijectively onto $\mathbb F_3$. Put $\Lambda_0(z)=1/(z+1)$, which maps $\infty,0,1$ to $0,1,-1$. Then $S:=\Lambda_0^{-1}\circ\Lambda$ permutes $\{\infty,0,1\}$, so $S$ is one of the six fractional linear maps in \eqref{eq:S3-action} below, and $f_E\circ\Lambda=(f_E\circ\Lambda_0)\circ S$. In characteristic $3$ one computes
\[
 f_E(\Lambda_0(z))=-\frac{a(z+1)^3}{z(z-1)}
 =-az+\frac{a}{z}+\frac{a}{z-1}-a,
\]
whose coefficient triple is $(-a,a,a)$. By \eqref{eq:S3-action}, each of the six substitutions maps the set $\{(-s,s,s):s\in k^\times\}$ to itself, and so does multiplication by $u^{-1}=\pm1$. Hence $(\alpha,\beta,\gamma)$ lies in this set.
\end{proof}

For later reference we record the action used in (c). Writing $F=\alpha z+\beta/z+\gamma/(z-1)$, a direct computation gives, modulo additive constants,
\begin{equation}\label{eq:S3-action}
\begin{array}{c|cccccc}
 S(z) & z & 1-z & 1/z & 1/(1-z) & z/(z-1) & (z-1)/z\\ \hline
 F\circ S & (\alpha,\beta,\gamma) & (-\alpha,-\gamma,-\beta) & (\beta,\alpha,-\gamma) & (-\beta,\gamma,-\alpha) & (\gamma,-\beta,\alpha) & (-\gamma,-\alpha,\beta)
\end{array}
\end{equation}

\begin{lemma}[The monomial exceptional family]\label{lem:monomial}
Let $f\in k[x]$ be a reduced polynomial of degree $d$, with $p\nmid d$ and $g(X_f)\ge2$, where $X_f:y^p-y=f(x)$. Suppose that $X_f$ is birational to $y^p-y=x^{-\lambda}$ for some $\lambda\mid p+1$. Then $\lambda=d$, and:
\begin{enumerate}[label=\textup{(\alph*)}]
\item if $d<p$, then $f=e(x+b)^d+e_0$ with $e\in k^\times$ and $b,e_0\in k$; in particular, if the coefficient of $x^{d-1}$ in $f$ vanishes, then $f-f(0)$ is a monomial;
\item if $d=p+1$, then $f-f(0)=e_1x^{p+1}+e_2x$ for some $e_1,e_2\in k$.
\end{enumerate}
\end{lemma}

\begin{proof}
Comparing genera gives $\lambda=d$. The substitution $x\mapsto1/x$ identifies $y^p-y=x^{-d}$ with $X_{x^d}:y^p-y=x^d$, a $p$-cyclic cover of the affine line. By \cite[Proposition~3.3]{LehrMatignon}, an isomorphism $X_f\to X_{x^d}$ can be modified by an automorphism of $X_f$ so that it descends to an affine map $x\mapsto ax+b$; the resulting data satisfy
\[
 c\,f(x)-(ax+b)^d=P^p-P,\qquad a\in k^\times,\ b\in k,\ c\in\Fp^\times,\ P\in k[x].
\]
(a) If $d<p$, the left side has degree at most $d<p$, so $P$ is constant. Hence $f=c^{-1}(ax+b)^d+\text{const}$. The coefficient of $x^{d-1}$ is $c^{-1}da^{d-1}b$, which vanishes only if $b=0$.

(b) If $d=p+1$, then $(ax+b)^{p+1}=a^{p+1}x^{p+1}+a^pbx^p+ab^px+b^{p+1}$. Choose $\beta$ with $\beta^p=a^pb$; then $a^pbx^p=(\beta x)^p-\beta x+\beta x$, and
\[
 c f(x)-a^{p+1}x^{p+1}-(\beta+ab^p)x-b^{p+1}=(P+\beta x)^p-(P+\beta x).
\]
The left side is reduced, so by Lemma~\ref{lem:reduced-unique} both sides are constant, and the claim follows.
\end{proof}

\begin{remark}\label{rem:DGLMS-alternative}
Lemmas~\ref{lem:E-rigid} and~\ref{lem:monomial} can also be deduced from the standard-form theory of \cite{DGLMS}: Lemma~3.6 there shows that every standard form of $y^p-y=a/(x^p-x)$ arises from the displayed model by a change of variables of the form in Lemma~\ref{lem:lift} (its automorphism group is $((C_p\times C_p)\rtimes C_{p-1})\rtimes C_2$), Lemma~3.8 there shows that the only one-pole standard form of $y^p-y=x^{-\lambda}$ is $y^p-y=x^\lambda$, and Proposition~4.15 there lists the standard forms in the genus-$4$, three-pole stratum for $p=3$. The arguments above avoid the standard-form normalization and use only \cite{Nakajima} and \cite{LehrMatignon}.
\end{remark}

\begin{remark}\label{rem:CS}
A Castelnuovo--Severi argument gives an elementary substitute for Proposition \ref{prop:VM} when $g>(p-1)^2$: two degree-$p$ maps $X\to\PP$ not differing by an automorphism of $\PP$ would map $X$ birationally onto a curve of bidegree $(p,p)$ in $\PP\times\PP$, whose arithmetic genus is $(p-1)^2$; see \cite[Section~3.11]{Stichtenoth}. This is the prime-degree version of \cite[Lemma~5]{PoonenI}. It is not used below, because it misses all genera $g\le(p-1)^2$, including most of the sharp boundary cases.
\end{remark}

\section{One branch point: the sharp \texorpdfstring{$p$}{p}-rank-zero range}\label{sec:one-pole}

If $\gamma(X)=0$, Proposition~\ref{prop:formulas} shows that an Artin--Schreier presentation has exactly one pole. After moving it to $\infty$ and reducing the equation, we may write
\[
 y^p-y=f(x),\qquad f\in k[x],\qquad \deg f=d,
\]
with $p\nmid d$, and then $N=d-1$. Thus $p\nmid N+1$ is necessary. We first isolate the remaining small obstructions.

\begin{proposition}[Small one-pole obstructions]\label{prop:zero-obstructions}
Let $p$ be odd and let $X/k$ be a $p$-rank-zero Artin--Schreier curve with
\[
 g(X)=\frac{p-1}{2}N.
\]
If $N=1$, or if $N=2$ \textup{(}which requires $p\ne3$\textup{)}, or if $(p,N)=(3,3)$, then
\[
 \Aut_k(X)\supsetneq C_p.
\]
\end{proposition}

\begin{proof}
Move the unique branch point to $\infty$ and use a reduced polynomial equation $y^p-y=f(x)$. By Lemma~\ref{lem:lift} (whose converse direction requires no hypothesis), it suffices to exhibit $(M,u,h)$ with $M\ne1$ satisfying \eqref{eq:lift}.

If $N=1$, then $d=2$. Since $p$ is odd, an affine translation of $x$ and an Artin--Schreier translation of $y$ put the equation in the form
\[
 y^p-y=ax^2,
\qquad a\ne0.
\]
Then $(x,y)\mapsto(-x,y)$ is an involution outside the deck group.

If $N=2$, then $d=3$ and hence $p\ne3$, so $p\ge5$. Translating $x$ removes the quadratic term, and translating $y$ removes the constant term, so over $k$ the equation has the form
\[
 y^p-y=ax^3+bx.
\]
Then $(x,y)\mapsto(-x,-y)$ is an involution outside $C_p$.

It remains to consider $(p,N)=(3,3)$, so $d=4$. In reduced form we may write
\[
 f(x)=a_4x^4+a_2x^2+a_1x+a_0,
\qquad a_4\ne0.
\]
For a translation $x\mapsto x+B$ one has, in characteristic $3$,
\[
 f(x+B)-f(x)
 =a_4B x^3+(a_4B^3-a_2B)x+\text{constant}.
\]
Choose a nonzero root $T\in k$ of
\[
 -a_4^3T^4+a_2^3T-a_4=0;
\]
such a root exists because the polynomial is nonconstant and its constant term is nonzero. Choose $B\in k^\times$ with $B^2=T$ and set
\[
 c=-a_4B^3+a_2B.
\]
Then $c^3=-a_4^3B^9+a_2^3B^3=B(-a_4^3T^4+a_2^3T)=a_4B$. Hence
\[
 f(x+B)-f(x)=c^3x^3-cx+K
\]
for some $K\in k$. Choose $\delta$ with $\delta^3-\delta=K$. Then
\[
 (x,y)\longmapsto(x+B,\,y+cx+\delta)
\]
is an automorphism of $X$. Because $B\ne0$, it is not a deck transformation.
\end{proof}

We next record the criterion that excludes the exceptional families for all one-pole models below.

\begin{lemma}[One-pole non-exceptionality]\label{lem:one-pole-nonexc}
Let $f\in\Fp[x]$ be a reduced polynomial of degree $d=N+1\ge4$ with $p\nmid d$, and assume one of the following:
\begin{enumerate}[label=\textup{(\roman*)}]
\item $d\nmid p+1$;
\item $d<p$, the coefficient of $x^{d-1}$ in $f$ vanishes, and $f-f(0)$ is not a monomial;
\item $d=p+1$ and $f-f(0)\notin k\,x^{p+1}+k\,x$.
\end{enumerate}
Then $X_f:y^p-y=f(x)$ is not birational to any curve in \eqref{eq:exceptional-families}. Consequently every geometric automorphism of $X_f$ induces an affine map $x\mapsto Ax+B$, and there are $u\in\Fp^\times$ and $h\in k[x]$ with
\begin{equation}\label{eq:poly-lift}
 f(Ax+B)-u f(x)=h(x)^p-h(x).
\end{equation}
If $h$ is nonconstant, then $\deg h\le d/p$. Hence, for every exponent $e>d/p$ with $p\nmid e$, the coefficient of $x^e$ on the right side of \eqref{eq:poly-lift} is zero.
\end{lemma}

\begin{proof}
The first and third families have positive $p$-rank, while $\gamma(X_f)=0$. Suppose $X_f$ were birational to $y^p-y=x^{-\lambda}$ with $\lambda\mid p+1$. Since $g(X_f)=\frac{p-1}2(d-1)\ge3$, Lemma~\ref{lem:monomial} gives $\lambda=d$, so $d\mid p+1$ and (i) fails. If (ii) holds, Lemma~\ref{lem:monomial}(a) shows that $f-f(0)$ is a monomial, a contradiction; if (iii) holds, Lemma~\ref{lem:monomial}(b) gives a contradiction. Hence Proposition~\ref{prop:VM} applies. By Lemma~\ref{lem:jumps} the induced point map fixes $\infty$, so $M$ is affine, and Lemma~\ref{lem:lift} gives \eqref{eq:poly-lift} with $h\in k(x)$. Since the left side has no finite poles, neither can $h$ (a pole of $h$ would give a pole of order divisible by $p$ in $h^p-h$); hence $h\in k[x]$. If $h$ is nonconstant, then $p\deg h=\deg(h^p-h)\le d$. The monomials of $h^p-h$ have degree divisible by $p$ or at most $\deg h\le d/p$.
\end{proof}

The next lemma settles all admissible values below the uniform range $N\ge7$.

\begin{lemma}[Small one-pole constructions]\label{lem:small-one-pole}
For every admissible pair with $3\le N\le6$, except $(p,N)=(3,3)$, the following reduced equation $y^p-y=f(x)$ defines a curve over $\Fp$ with full geometric automorphism group $C_p$:
\[
\begin{array}{c|c|c}
N & p & f(x)\\ \hline
3 & p\ge5 & x^4+x^2+x\\
4 & p=3 & x^5+x^4+x^2\\
4 & p\ge7 & x^5+x^3+x^2\\
5 & p=5 & x^6+x^4+x\\
5 & p\ge7 & x^6+x^4+x^3\\
6 & p=3 & x^7+x^4\\
6 & p=5 & x^7+x^6\\
6 & p\ge11 & x^7+x^5+x^4.
\end{array}
\]
\end{lemma}

\begin{proof}
The table lists every admissible pair: for $3\le N\le6$ the condition $p\nmid N+1$ excludes exactly $(p,N)\in\{(5,4),(3,5),(7,6)\}$, and $(3,3)$ is excluded by Proposition~\ref{prop:zero-obstructions}. In every row $d=N+1$, all exponents are prime to $p$, and there is a unique branch point at $\infty$.

\emph{Non-exceptionality.} For the rows $(p,N)=(3,4),(3,6),(5,6)$ one has $d\nmid p+1$, so Lemma~\ref{lem:one-pole-nonexc}(i) applies. In the rows with $p\ge5$, $p\ge7$ or $p\ge11$, either $d\nmid p+1$, or $d\mid p+1$ with $d<p+1$; in the latter case $d<p$, the coefficient of $x^{d-1}$ vanishes and $f-f(0)$ is not a monomial, so Lemma~\ref{lem:one-pole-nonexc}(ii) applies. For $(p,N)=(5,5)$ one has $d=6=p+1$ and $f-f(0)=x^6+x^4+x$ has an $x^4$-term, so Lemma~\ref{lem:one-pole-nonexc}(iii) applies. Hence in every row \eqref{eq:poly-lift} holds for every geometric automorphism.

\emph{Rigidity.} For $N=3$ and $p\ge5$, take $f=x^4+x^2+x$. Here $4<p$, so $h$ is constant. The coefficients of $x^4$ and $x^3$ in \eqref{eq:poly-lift} give $u=A^4$ and $4A^3B=0$, so $B=0$; the coefficients of $x^2$ and $x$ then give
\[
 A^2=u,\qquad A=u.
\]
Thus $A^2=A^4=A$, and $A=1$.

For $(p,N)=(3,4)$, take $f=x^5+x^4+x^2$. Now $h=cx+\delta$. Modulo $3$ the coefficients of $x^5,x^4,x^2$ give
\[
 u=A^5,
 \qquad
 A^4(2B+1)=u,
 \qquad
 A^2(B^3+1)=u.
\]
Thus $A=1+2B$ and $A^3=B^3+1$. Cubing the first equality in characteristic $3$ gives $A^3=1+2B^3$, so $B=0$ and $A=1$.

For $N=4$ and $p\ge7$, take $f=x^5+x^3+x^2$. Since $5<p$, $h$ is constant. The coefficients of $x^5,x^4,x^3,x^2$ give
\[
 u=A^5,\qquad 5A^4B=0,\qquad A^3=u,\qquad A^2=u,
\]
the last two after using $B=0$. Hence $A^2=A^3=A^5$, and $A=1$.

For $(p,N)=(5,5)$, take $f=x^6+x^4+x$ and write $h=cx+\delta$. The $x^6$- and $x^4$-coefficients give
\[
 u=A^6,\qquad A^4=u,
\]
so $A^2=1$. Since $\binom63\equiv0\pmod5$, the coefficient of $x^3$ is $4A^3B$, hence $B=0$. The coefficient of $x^5$ is then $0=c^5$, so $c=0$; comparison at $x$ gives $A=u=A^6$, hence $A^5=1$. Together with $A^2=1$ this yields $A=1$.

For $N=5$ and $p\ge7$, take $f=x^6+x^4+x^3$. Since $6<p$, $h$ is constant. The coefficients of $x^6,x^5,x^4,x^3$ give
\[
 u=A^6,\qquad 6A^5B=0,\qquad A^4=u,\qquad A^3=u,
\]
so $B=0$ and $A=1$.

For $(p,N)=(3,6)$, take $f=x^7+x^4$ and write
\[
 h=c_2x^2+c_1x+\delta.
\]
By Lucas' theorem $\binom72\equiv\binom42\equiv0\pmod3$, so the coefficient of $x^2$ on the left side of \eqref{eq:poly-lift} is zero, while on the right side it is $-c_2$; thus $c_2=0$. The coefficient of $x^6$ is $A^6B=c_2^3=0$, hence $B=0$. Finally $u=A^7$ from degree $7$, while degree $4$ gives $A^4=u$; therefore $A^3=1$, and in characteristic $3$ this forces $A=1$.

For $(p,N)=(5,6)$, take $f=x^7+x^6$ and write $h=cx+\delta$. Since $\binom72\equiv1$ and $\binom62\equiv0\pmod5$, the coefficient of $x^2$ on the left side is $A^2B^5$, whereas $h^5-h$ has no $x^2$-term. Hence $B=0$. The degree-$7$ and degree-$6$ coefficients then give $u=A^7=A^6$, so $A=1$.

Finally, for $N=6$ and $p\ge11$, take $f=x^7+x^5+x^4$. Since $7<p$, $h$ is constant. The coefficients of $x^7,x^6,x^5,x^4$ give
\[
 u=A^7,\qquad 7A^6B=0,\qquad A^5=u,\qquad A^4=u,
\]
so again $B=0$ and $A=1$.

Thus in every row the induced affine transformation is the identity, and Lemma~\ref{lem:lift} shows that the full geometric automorphism group is the deck group $C_p$.
\end{proof}

We now treat the uniform large range.

\begin{proposition}[Large one-pole construction]\label{prop:one-pole-large}
Let $p$ be an odd prime and let $N\ge7$ satisfy $p\nmid N+1$. Then there exists an Artin--Schreier curve $X/\Fp$ such that
\[
 g(X)=\frac{p-1}{2}N,\qquad \gamma(X)=0,
\qquad \Aut_{\Fpb}(X)\cong C_p.
\]
\end{proposition}

\begin{proof}
Set $d=N+1\ge8$. In all cases below the polynomial $f$ has the form $x^d+x^{d-s}+x^{d-t}$ with $2\le s<t\le5$ and all exponents prime to $p$. Since $d\ge8$ and $p\ge3$, we have $d-5>d/p$, so the last assertion of Lemma~\ref{lem:one-pole-nonexc} applies to every exponent $e\ge d-5$ prime to $p$.

\smallskip
\noindent\emph{Case 1: $d\not\equiv1\pmod p$.}
Choose consecutive integers $s<t$ with $2\le s<t\le5$ such that
\[
 s\not\equiv d\pmod p,\qquad t\not\equiv d\pmod p.
\]
For $p=3$ one has $d\equiv2\pmod3$ and we take $(s,t)=(3,4)$. For $p\ge5$, take $(s,t)=(3,4)$ if $d\equiv2\pmod p$, $(s,t)=(4,5)$ if $d\equiv3\pmod p$, and $(s,t)=(2,3)$ otherwise. Define
\begin{equation}\label{eq:oneA-new}
 X:\qquad y^p-y=x^d+x^{d-s}+x^{d-t}.
\end{equation}
All exponents are positive and prime to $p$.

\emph{Non-exceptionality.} If $d\mid p+1$, then $d\ne p+1$ because $d\not\equiv1\pmod p$, so $d<p$. The coefficient of $x^{d-1}$ vanishes because $s\ge2$, and $f-f(0)$ is not a monomial. Hence Lemma~\ref{lem:one-pole-nonexc} applies in all cases, through (i) or (ii).

\emph{Rigidity.} The coefficient of $x^d$ in \eqref{eq:poly-lift} gives $u=A^d$. Since $d-1>d/p$ and $p\nmid d-1$, comparison at degree $d-1$ gives
\[
 dA^{d-1}B=0,
\]
so $B=0$. Comparing at degrees $d-s$ and $d-t$ gives
\[
 A^{d-s}=u=A^d,\qquad A^{d-t}=u=A^d.
\]
Thus $A^s=A^t=1$, and because $t=s+1$ we obtain $A=1$.

\smallskip
\noindent\emph{Case 2: $d\equiv1\pmod p$.}
Take
\begin{equation}\label{eq:oneB-new}
 X:\qquad y^p-y=x^d+x^{d-2}+x^{d-5}.
\end{equation}
The exponents are positive and prime to $p$, since $d-2\equiv-1$ and $d-5\equiv-4\pmod p$.

\emph{Non-exceptionality.} If $d\mid p+1$, then $d\equiv1\pmod p$ forces $d=p+1$, and $f-f(0)$ contains the term $x^{p-1}$ with $p-1\ge2$. Hence Lemma~\ref{lem:one-pole-nonexc} applies, through (i) or (iii).

\emph{Rigidity.} The coefficient of $x^d$ gives $u=A^d$. Since
\[
 \binom d2\equiv0\pmod p,
\]
comparison at degree $d-2$ gives $A^{d-2}=A^d$, hence
\begin{equation}\label{eq:A2}
 A^2=1.
\end{equation}
If $p\ge5$, then $\binom d3\equiv0\pmod p$, and comparison at degree $d-3$ gives
\[
 (d-2)A^{d-3}B=0,
\]
so $B=0$. Comparison at degree $d-5$ gives $A^5=1$, and together with \eqref{eq:A2} this yields $A=1$.

It remains to consider $p=3$. We distinguish the three residue classes of $d$ modulo $9$; recall that $d\equiv1\pmod3$, and all exponents used below are prime to $3$ and larger than $d/3$.

If $d\equiv1\pmod9$, Lucas' theorem gives $\binom d3\equiv0\pmod3$. The coefficient of $x^{d-3}$ in
$f(Ax+B)-u f(x)$ is therefore
\[
 \binom d3 A^{d-3}B^3+(d-2)A^{d-3}B
   =-A^{d-3}B,
\]
so $B=0$.

If $d\equiv4\pmod9$ (so $d\ge13$), Lucas' theorem gives
\[
 \binom d6\equiv0,
 \qquad
 \binom{d-2}{4}\equiv0
 \pmod3.
\]
Hence the coefficient of $x^{d-6}$ is
\[
 \binom d6 A^{d-6}B^6
 +\binom{d-2}{4}A^{d-6}B^4
 +(d-5)A^{d-6}B
 =-A^{d-6}B,
\]
and again $B=0$.

Finally, suppose $d\equiv7\pmod9$. Lucas' theorem gives
\[
 \binom d3\equiv2,
 \qquad
 d-2\equiv2,
 \qquad
 \binom d5\equiv0,
 \qquad
 \binom{d-2}{3}\equiv1
 \pmod3.
\]
Thus the coefficients of $x^{d-3}$ and $x^{d-5}$ are, respectively,
\[
 2A^{d-3}(B^3+B)
 \quad\text{and}\quad
 A^{d-5}(B^3+1-A^5).
\]
They must vanish, and hence
\[
 B(B^2+1)=0,
 \qquad
 B^3+1-A^5=0.
\]
If $B\ne0$, then $B^2=2$. If $A=1$, the second equation gives $B=0$, a contradiction. If $A=-1$, it gives $B^3=1$, hence $B=1$, again contradicting $B^2=2$. Thus $B=0$ in every case. The coefficient of $x^{d-5}$ then gives $A^5=1$, and \eqref{eq:A2} yields $A=1$.

In both cases every geometric automorphism induces the identity on the $x$-line, and hence belongs to the deck group. Proposition \ref{prop:formulas} gives
\[
 g(X)=\frac{p-1}{2}(d-1)=\frac{p-1}{2}N,
 \qquad \gamma(X)=0.
\]
\end{proof}

\begin{theorem}[Sharp $p$-rank-zero range]\label{thm:zero-rank}
Let $p$ be an odd prime and let $N\ge1$. There exists an Artin--Schreier curve $X/\Fp$ such that
\[
 g(X)=\frac{p-1}{2}N,\qquad
 \gamma(X)=0,\qquad
 \Aut_{\Fpb}(X)\cong C_p
\]
if and only if
\[
 p\nmid N+1,\qquad N\ge3,\qquad (p,N)\ne(3,3).
\]
\end{theorem}

\begin{proof}
For necessity, let $d$ be the lower jump at the unique branch point. Proposition~\ref{prop:formulas} gives $d=N+1$, and reduced Artin--Schreier form gives $p\nmid d$. Proposition~\ref{prop:zero-obstructions} excludes $N=1$, the admissible case $N=2$, and $(p,N)=(3,3)$.

Conversely, suppose the displayed conditions hold. If $3\le N\le6$, Lemma~\ref{lem:small-one-pole} gives an explicit curve over $\Fp$. If $N\ge7$, Proposition~\ref{prop:one-pole-large} applies.
\end{proof}

\section{Two branch points}\label{sec:two-poles}

We begin with a simple arithmetic lemma.

\begin{lemma}\label{lem:two-arithmetic}
Let $p$ be an odd prime and $N\ge3$. Except for $(p,N)=(3,4)$, there exist positive integers $a,b,t$ such that
\[
 a+b=N,
 \qquad p\nmid abt,
 \qquad a\ne b,
 \qquad 1\le t<b,
 \qquad \gcd(a+t,b-t)=1.
\]
\end{lemma}

\begin{proof}
If $N\not\equiv2\pmod p$ and $N\ne4$, take
\[
 (a,b,t)=(N-2,2,1).
\]
If $N\equiv2\pmod p$ and $N\ne2p+2$, take
\[
 (a,b,t)=(N-p-1,p+1,1).
\]
Here $N>2$ implies $N\ge p+2$, and
\[
 \gcd(a+t,b-t)=\gcd(N-p,p)=1.
\]
If $N=2p+2$ and $p\ge5$, take
\[
 (a,b,t)=(2p-1,3,2).
\]
If $(p,N)=(3,8)$, take $(a,b,t)=(1,7,2)$. Finally, if $N=4$ and $p\ge5$, take $(a,b,t)=(1,3,2)$. Since $N=4\not\equiv2\pmod p$ for odd $p$, these cases exhaust all possibilities.
\end{proof}

\begin{proposition}\label{prop:two-poles}
Let $p$ be an odd prime and let $N\ge3$. Then there exists an Artin--Schreier curve $X/\Fp$ with
\[
 g(X)=\frac{p-1}{2}N,
 \qquad
 \gamma(X)=p-1,
 \qquad
 \Aut_{\Fpb}(X)\cong C_p.
\]
\end{proposition}

\begin{proof}
Assume first that $(p,N)\ne(3,4)$, and choose $a,b,t$ as in Lemma \ref{lem:two-arithmetic}. Consider
\begin{equation}\label{eq:two-new}
 X:\qquad y^p-y=x^b+x^t+x^{-a}.
\end{equation}
The poles are $0$ and $\infty$, with distinct lower jumps $a$ and $b$. Proposition \ref{prop:formulas} gives
\[
 g(X)=\frac{p-1}{2}(a+b)=\frac{p-1}{2}N,
 \qquad
 \gamma(X)=p-1.
\]
Here $(r,N)=(1,N)$ with $N\ge3$, so Lemma~\ref{lem:exceptional-invariants} shows that $X$ is not exceptional (note $p-1\ne1$). Thus every automorphism descends to the base.

Because the two lower jumps are distinct, Lemma \ref{lem:jumps} forces the induced base automorphism to fix $0$ and $\infty$ separately. Hence it is $x\mapsto\lambda x$ for some $\lambda\in k^\times$. Equation \eqref{eq:lift} gives
\[
 (\lambda^b-u)x^b+(\lambda^t-u)x^t+(\lambda^{-a}-u)x^{-a}=h^p-h.
\]
A pole of $h^p-h$ has order divisible by $p$, whereas $a,b,t$ are prime to $p$. Hence the pole of the left side at $0$ must vanish, giving $\lambda^{-a}=u$, and at $\infty$ cancellation first at order $b$ and then at order $t$ gives
\[
 \lambda^b=u,\qquad \lambda^t=u.
\]
Therefore
\[
 \lambda^{a+t}=1,
 \qquad
 \lambda^{b-t}=1.
\]
By the gcd condition, $\lambda=1$. Hence every automorphism lies in the deck group.

It remains to treat $(p,N)=(3,4)$. Take
\begin{equation}\label{eq:special34}
 X:\qquad y^3-y=x^2+x+x^{-2}.
\end{equation}
Then $g(X)=4$ and $\gamma(X)=2$, i.e.\ $(r,N)=(1,4)$, so $X$ is not exceptional and every automorphism descends. Since the two lower jumps are both $2$, the induced transformation either has the form $x\mapsto\lambda x$ or $x\mapsto\lambda/x$.

In both cases the left side of \eqref{eq:lift} has pole order at most $2$ at $0$ and $\infty$ and no other poles, while a pole of $h^3-h$ has order divisible by $3$. Thus $h^3-h$ is constant. If $x\mapsto\lambda x$, comparison of the coefficients of $x^2$ and $x$ gives $\lambda^2=u$ and $\lambda=u$, so $\lambda=1$. If $x\mapsto\lambda/x$, then $f(\lambda/x)$ contains the term $\lambda x^{-1}$, whereas $u f(x)$ has no $x^{-1}$-term, a contradiction. Thus no inversion lifts, and the full automorphism group is $C_3$.
\end{proof}

\begin{proposition}\label{prop:N2-impossible}
Let $p$ be odd. If an Artin--Schreier curve has exactly two branch points, both with lower jump $1$, then its geometric automorphism group strictly contains the deck group $C_p$.
\end{proposition}

\begin{proof}
Over $k$ move the two branch points to $0$ and $\infty$. After an Artin--Schreier translation in $y$, the equation has the form
\[
 y^p-y=ax+\frac{b}{x},\qquad a,b\in k^\times.
\]
The involution
\[
 x\longmapsto\frac{b}{ax},\qquad y\longmapsto y
\]
preserves the right-hand side, so it defines an automorphism by the converse part of Lemma~\ref{lem:lift}. It does not fix $k(x)$ pointwise, so it is not in the deck group.
\end{proof}

\section{At least three branch points: the sharp range}\label{sec:many-poles}

Let $r\ge2$ and put
\[
 m:=N-2r.
\]
An Artin--Schreier curve of $p$-rank $r(p-1)$ has $r+1$ poles with jumps $d_1,\ldots,d_{r+1}\ge1$, and
\[
 m=\sum_{i=1}^{r+1}(d_i-1).
\]
In particular $m\ge0$ is necessary. We prove that every $m\ge0$ occurs with full geometric automorphism group $C_p$.

We first record three elementary choices of squarefree polynomials. They will be used only to destroy projective symmetries of the simple branch points.

\begin{lemma}[Asymmetric polynomial choices]\label{lem:asym-poly}
Let $p$ be odd.
\begin{enumerate}[label=\textup{(\alph*)}]
\item For every $n\ge3$ there is a monic squarefree polynomial $A_n\in\Fp[x]$ of degree $n$ whose zero set has trivial stabilizer in the affine group $x\mapsto ax+b$.
\item For every $n\ge1$ there is a squarefree polynomial $B_n\in\Fp[x]$ of degree $n$, with $B_n(0)\ne0$, whose zero set has trivial stabilizer under scalings $x\mapsto ax$.
\item For every $n\ge3$ there is a monic squarefree polynomial $C_n\in\Fp[x]$ of degree $n$ such that the zero set of $C_n$ has trivial affine stabilizer and, for
\[
 \omega_n=d\!\left(\frac1{C_n(x)}\right),
\]
the point $\infty$ is the unique zero of $\omega_n$ of maximal multiplicity.
\end{enumerate}
\end{lemma}

\begin{proof}
If an affine map preserves the zero set of a monic squarefree polynomial $P$ of degree $n$, then $P(ax+b)=a^nP(x)$, and similarly for scalings.

For (a), suppose first that $p\mid n$ and take
\[
 A_n=x^n+x^{n-1}+1.
\]
Then $A_n'=-x^{n-2}$ and $A_n(0)=1$, so $A_n$ is squarefree. In $A_n(ax+b)=a^nA_n(x)$, the coefficients of $x^{n-1}$ give $a^{n-1}=a^n$, so $a=1$; then, since $\binom n2\equiv0\pmod p$, the coefficients of $x^{n-2}$ give $(n-1)b=-b=0$.

Now suppose $p\nmid n$. Choose $c,d\in\Fp^\times$ such that
\[
 A_n=x^n+cx+d
\]
is squarefree. Such a pair exists. Indeed, if a common root $\alpha$ of $A_n$ and $A_n'$ exists, then
\[
 c=-n\alpha^{n-1},\qquad d=(n-1)\alpha^n.
\]
If $n\equiv1\pmod p$ this contradicts $d\ne0$; otherwise $\alpha=-nd/((n-1)c)\in\Fp^\times$, so at most $p-1$ of the $(p-1)^2$ pairs $(c,d)$ are bad. For a good pair, the $x^{n-1}$-coefficient of $A_n(ax+b)=a^nA_n(x)$ gives $na^{n-1}b=0$, so $b=0$ (here $n\ge3$), while the coefficients of $x$ and $1$ give $a^{n-1}=a^n=1$, hence $a=1$.

For (b), take $B_1=x-1$. For $n\ge2$ with $p\mid n$, take $B_n=x^n+x+1$, which is squarefree because $B_n'=1$. If $p\nmid n$, choose a squarefree $B_n=x^n+cx+d$ with $c,d\in\Fp^\times$ by the preceding counting argument. In either case $B_n(0)\ne0$, and from
\[
 B_n(ax)=a^nB_n(x)
\]
the coefficients of $x$ and $1$ give $a^{n-1}=a^n=1$, hence $a=1$ (for $n=1$ the constant term alone gives $a=1$).

For (c), note that $\omega_n=-C_n'\,C_n^{-2}\,dx$, so the finite zeros of $\omega_n$ are the zeros of $C_n'$, with the same multiplicities, and the order of $\omega_n$ at $\infty$ is $n-1$ if $p\nmid n$ and at least $n$ if $p\mid n$. If $p\mid n$, use $C_n=x^n+x^{n-1}+1$. Then
\[
 C_n'=-x^{n-2},
\]
so $\omega_n$ has a zero of order $n-2$ at $0$ and a zero of order $n$ at $\infty$; the latter is unique of maximal multiplicity. If $p\nmid n$ and $n-1$ is not a power of $p$, use a squarefree $C_n=x^n+cx+d$ as in (a). Write $n-1=p^e q$ with $q>1$ and $(q,p)=1$. The finite zeros of
\[
 C_n'=nx^{n-1}+c
\]
have multiplicity $p^e<n-1$, whereas $\omega_n$ has a zero of order $n-1$ at $\infty$. Finally, if $n-1=p^e$ with $e\ge1$, take
\[
 C_n=x^n+x^2+x.
\]
This polynomial is squarefree: $C_n'=x^{n-1}+2x+1$ does not vanish at $0$, and if $\alpha\ne0$ were a common root of $C_n/x=x^{n-1}+x+1$ and $C_n'$, subtraction would give $\alpha=0$. Its affine stabilizer is trivial by comparing the $x^{n-1}$, $x^2$, and $x$ coefficients (note $n\ge4$), and $C_n'$ is squarefree because $(C_n')'=2$. Thus all finite zeros of $\omega_n$ are simple while the zero at $\infty$ has order $n-1\ge3$. This proves (c).
\end{proof}

The only numerical boundary at which the first Valentini--Madan exceptional family can occur is the ordinary case $r=p-1$, $N=2p-2$. For $p\ge5$ we use the following explicit asymmetric $p$-point configuration.

\begin{lemma}\label{lem:p-point-set}
Let $p\ge5$ and
\[
 Q_p(x)=x^p+x^2+x+1.
\]
Then $Q_p$ is squarefree and the stabilizer of its zero set in $\PGL_2(\Fpb)$ is trivial.
\end{lemma}

\begin{proof}
Since $Q_p'(x)=2x+1$ and, for $a\in\Fp$, $Q_p(a)=(a+1)^2$, the only possible common root of $Q_p$ and $Q_p'$ would be $a=-1/2$, where $Q_p(a)=1/4\ne0$. Hence $Q_p$ is squarefree.

Let
\[
 F(X,Z)=X^p+X^2Z^{p-2}+XZ^{p-1}+Z^p.
\]
Since $F(1,0)\ne0$, the zero set of $F$ in $\PP$ is the zero set of $Q_p$. If $M=\begin{pmatrix}A&B\\ C&D\end{pmatrix}$ stabilizes it, then, since $F$ is squarefree,
\begin{equation}\label{eq:binary-stab}
 F(AX+BZ,CX+DZ)=\kappa F(X,Z)
\end{equation}
for some $\kappa\ne0$. Put $\Delta=AD-BC\ne0$. For $3\le j\le p-2$, only
\[
 (AX+BZ)^2(CX+DZ)^{p-2}
 \quad\text{and}\quad
 (AX+BZ)(CX+DZ)^{p-1}
\]
contribute to $X^jZ^{p-j}$. Using
\[
 \binom{p-1}{j}\equiv(-1)^j,
 \qquad
 \binom{p-2}{j}\equiv(-1)^j(j+1)
 \pmod p,
\]
and the analogous formulas with $j-1$ and $j-2$, the coefficient is
\begin{align*}
 (-1)^j C^{j-2}D^{p-j-2}
 \bigl(& (j+1)B^2C^2-2jABCD +(j-1)A^2D^2 \\
 &{}+BC^2D-ACD^2\bigr).
\end{align*}
Factoring the expression in parentheses gives
\begin{equation}\label{eq:middle-coeff}
 (-1)^j C^{j-2}D^{p-j-2}\Delta
 \bigl(j\Delta-(AD+BC+CD)\bigr),
\end{equation}
while the coefficient of $X^jZ^{p-j}$ on the right of \eqref{eq:binary-stab} is $0$.

Assume first $p\ge7$. If $C,D\ne0$, the cases $j=3$ and $j=4$ in \eqref{eq:middle-coeff} force simultaneously
\[
 3\Delta=AD+BC+CD=4\Delta,
\]
a contradiction. Hence $C=0$ or $D=0$. If $D=0$, then $B,C\ne0$ and the case $j=p-2$ gives the coefficient $\pm B^2C^{p-2}\ne0$, again a contradiction. Thus $C=0$. Scaling the matrix, take $D=1$. Comparing the coefficients of $X^p$, $X^2Z^{p-2}$, $XZ^{p-1}$, and $Z^p$ gives
\[
 \kappa=A^p=A^2,
 \qquad
 A(2B+1)=A^2,
 \qquad
 B^p+B^2+B+1=A^2.
\]
Thus $B=(A-1)/2$, and $B^p=(A^p-1)/2=(A^2-1)/2$; substituting into the last identity gives $(3A^2+1)/4=A^2$, i.e.\ $A^2=1$. Since also $A^{p-2}=1$ and $p-2$ is odd, $A=1$ and then $B=0$.

It remains to treat $p=5$. The coefficient of $X^3Z^2$ gives either $C=0$ or
\[
 AD+BC+CD=3\Delta.
\]
The first case is handled exactly as above. In the second case scale $C=1$. Since $\Delta=AD-B$, the displayed relation becomes
\[
 B=D(1-2A).
\]
Substituting this into the coefficient of $X^4Z$ gives
\[
 2AB+3A^2D+4AD+B
 =D(1-A-A^2)=0.
\]
Since $D=0$ would force $B=0$ and hence $\Delta=0$, we have
\[
 A^2+A-1=0.
\]
The coefficient of $X^2Z^3$ in the transformed form is
\[
 3B^2D+ABD^2+A^2D^3+BD^2+4AD^3
   =D^3(A^2+A+4)=0,
\]
where the last equality uses $A^2+A=1$. On the right-hand side of \eqref{eq:binary-stab}, however, the coefficient of $X^2Z^3$ is $\kappa\ne0$, a contradiction. Hence only the identity survives for $p=5$ as well.
\end{proof}

We also need one explicit curve at the small exceptional numerical point $(p,r,N)=(3,2,4)$.

\begin{lemma}\label{lem:p3-boundary}
Over $\mathbb F_3$, the Artin--Schreier curve
\begin{equation}\label{eq:p3-boundary}
 y^3-y=\frac{x^2}{x^3+2x^2+1}
\end{equation}
has genus $4$, $3$-rank $4$, and full geometric automorphism group $C_3$.
\end{lemma}

\begin{proof}
The cubic $q=x^3+2x^2+1$ has no root in $\mathbb F_3$, hence is irreducible and squarefree. Thus \eqref{eq:p3-boundary} has three simple poles and is regular at $\infty$, so Proposition \ref{prop:formulas} gives $g=4$ and $\gamma=4$.

Let $\alpha$ be a root of $q$, so $\alpha^3=\alpha^2+2$, and put $\beta=\alpha^3$, $\gamma=\alpha^9$. Then
\[
 q(x)=(x-\alpha)(x-\beta)(x-\gamma).
\]
The change of coordinate
\[
 z=(\alpha+2)\frac{x-\beta}{x-\alpha}
\]
sends $\alpha,\beta,\gamma$ to $\infty,0,1$, respectively, and its inverse is
\[
 x=\frac{\alpha z-(\alpha+2)\beta}{z-(\alpha+2)}.
\]
Consequently the partial-fraction expansion of $x^2/q(x)$ has the form
\[
 \frac{x^2}{q(x)}=az+\frac b z+\frac c{z-1}+\delta
 \qquad(\delta\in\mathbb F_{27}).
\]
The three principal coefficients can be read off directly at the poles:
\begin{align*}
 a&=\frac{\alpha^2}
 {(\alpha+2)(\alpha-\beta)^2(\alpha-\gamma)},\\
 b&=\frac{(\alpha+2)\beta^2}
 {(\beta-\alpha)^2(\beta-\gamma)},\\
 c&=\frac{(\alpha+2)(\beta-\alpha)\gamma^2}
 {(\gamma-\alpha)^3(\gamma-\beta)}.
\end{align*}
Using $\alpha^3=\alpha^2+2$ these simplify, in the basis $1,\alpha,\alpha^2$ of $\mathbb F_{27}$, to
\[
 a=\alpha+2\alpha^2,
 \qquad
 b=1+2\alpha+\alpha^2,
 \qquad
 c=2+2\alpha+\alpha^2.
\]
Since we work over $k=\Fpb$, the constant $\delta$ is removed by replacing $y$ by $y+t$ with $t^3-t=\delta$.

\emph{Non-exceptionality.} The pair $(r,N)=(2,4)$ excludes the second and third families. If the curve were birational to $y^3-y=a'/(x^3-x)$, Lemma~\ref{lem:E-rigid}(c) would give $b=c$. Since $b\ne c$, Proposition \ref{prop:VM} applies.

\emph{Rigidity.} Every geometric automorphism induces a fractional linear transformation permuting $\{0,1,\infty\}$, and since all poles are simple, $h$ is constant in \eqref{eq:lift}. By \eqref{eq:S3-action} the six such transformations act on the coefficient triple by
\[
 (a,b,c),\ (-a,-c,-b),\ (b,a,-c),\ (-b,c,-a),\ (c,-b,a),\ (-c,-a,b),
\]
and a transformation lifts only if its triple equals $u(a,b,c)$ with $u=\pm1$. The six elements
\[
 a,-a,b,-b,c,-c
\]
of $\mathbb F_{27}$ are pairwise distinct. Thus the four triples beginning with $b,-b,c,-c$ cannot equal $u(a,b,c)$. The remaining nonidentity triple begins with $-a$, so it would force $u=-1$; its second component is then $-c$, whereas the second component of $-(a,b,c)$ is $-b$, and $b\ne c$. Hence none of the five nonidentity permutations lifts, and every geometric automorphism belongs to the deck group $C_3$.
\end{proof}

\begin{lemma}\label{lem:triple}
Let $p$ be an odd prime.
\begin{enumerate}[label=\textup{(\alph*)}]
\item If $p\ge5$ and $m\ge4$, there exist integers $a,b,c\ge2$, prime to $p$, not all equal, such that
\[
 a+b+c=m+3.
\]
\item If $p=3$ and $m\ge5$, the same conclusion holds.
\end{enumerate}
\end{lemma}

\begin{proof}
Assume first that $p=3$. If $m\not\equiv1\pmod3$, take $(a,b,c)=(2,2,m-1)$. If $m\equiv1\pmod3$, take $(a,b,c)=(2,4,m-3)$. For $m\ge5$ all entries are at least $2$, prime to $3$, and not all equal.

Now let $p\ge5$. If $p\nmid m-1$, take $(a,b,c)=(2,2,m-1)$. If $p\mid m-1$, take $(a,b,c)=(2,3,m-2)$. Again the entries have the required properties for $m\ge4$.
\end{proof}

\begin{proposition}\label{prop:many-poles}
Let $p$ be an odd prime and let $r\ge2$. For every integer $N\ge2r$ there exists an Artin--Schreier curve $X/\Fp$ such that
\[
 g(X)=\frac{p-1}{2}N,
 \qquad
 \gamma(X)=r(p-1),
 \qquad
 \Aut_{\Fpb}(X)\cong C_p.
\]
\end{proposition}

\begin{proof}
Put $m=N-2r\ge0$. In each case below, $f$ has exactly $r+1$ poles, all of order prime to $p$, with $\sum(d_i-1)=m$; hence Proposition \ref{prop:formulas} gives the required genus and $p$-rank.

\smallskip
\noindent\emph{Case 1: $m=0$.}
All $r+1$ lower jumps must be $1$. First suppose $r\ne p-1$. Put $n=r+1\ge3$ and choose $C_n$ as in Lemma \ref{lem:asym-poly}(c). Set
\[
 f(x)=\frac1{C_n(x)}.
\]
The curve has $n$ simple poles. Since $(r,N)=(r,2r)$ with $r\ge2$ and $r\ne p-1$, it is not exceptional by Lemma~\ref{lem:exceptional-invariants}. Thus every automorphism descends. Since both $f(M(x))$ and $f(x)$ have only simple poles, equation \eqref{eq:lift} forces $h$ to be constant. Differentiating gives
\[
 M^*df=u\,df.
\]
Hence the point map of $M$ preserves the zeros of $df=\omega_n$ with their multiplicities. By Lemma \ref{lem:asym-poly}(c), the unique zero of maximal multiplicity is $\infty$, so $M$ fixes $\infty$ and is affine. Its point action preserves the zeros of $C_n$, whose affine stabilizer is trivial. Hence $M=1$.

Now let $r=p-1$. If $p\ge5$, take
\[
 f(x)=\frac1{Q_p(x)},
 \qquad Q_p=x^p+x^2+x+1.
\]
The second and third exceptional families are excluded by $(r,N)=(p-1,2p-2)$ with $p\ge5$. By Lemma \ref{lem:p-point-set}, the branch set has trivial projective stabilizer, so Lemma~\ref{lem:E-rigid}(b) excludes the first family. Every automorphism therefore descends and must induce a projective transformation stabilizing the branch set, hence the identity; so it lies in $C_p$. The remaining possibility $p=3$, $r=2$ is exactly Lemma \ref{lem:p3-boundary}.

For the remaining cases $m>0$, hence $N>2r$ with $r\ge2$, and Lemma~\ref{lem:exceptional-invariants} shows that none of the Valentini--Madan exceptional families can occur. Proposition~\ref{prop:VM} applies throughout Cases~2--4.

\smallskip
\noindent\emph{Case 2: $m=1$.}
There is one jump $2$ and $r$ simple jumps. If $r=2$, take
\[
 f(x)=x^2+x+\frac1x+\frac1{x-1}.
\]
The unique jump-$2$ point is $\infty$, so a descended base transformation is affine and preserves $\{0,1\}$: it is either the identity or $x\mapsto1-x$. In the latter case the leading term gives $u=1$, while
\[
 f(1-x)-f(x)=-4x+2-\frac2x-\frac2{x-1}
\]
has simple poles, impossible for $h^p-h$. Hence only the identity lifts.

If $r\ge3$, choose $A_r$ as in Lemma \ref{lem:asym-poly}(a) and set
\[
 f(x)=x^2+\frac1{A_r(x)}.
\]
The jump-$2$ point $\infty$ is fixed, so every descended transformation is affine; the simple branch set has trivial affine stabilizer, and again $M=1$.

\smallskip
\noindent\emph{Case 3: $m=2$.}
Assume first $p\ge5$. If $r=2$, take
\[
 f(x)=x^3+\frac1x+\frac1{x-1}.
\]
The only possible nontrivial base transformation is $x\mapsto1-x$. The leading term gives $u=-1$, while
\[
 f(1-x)+f(x)=1-3x+3x^2
\]
is a nonconstant polynomial of degree $2<p$, so it cannot be $h^p-h$. If $r\ge3$, choose $A_r$ as in Lemma \ref{lem:asym-poly}(a) and use
\[
 f(x)=x^3+\frac1{A_r(x)}.
\]
The unique jump-$3$ point $\infty$ is fixed and the affine stabilizer of the simple branch set is trivial.

Now let $p=3$. Then the jump $3$ is not allowed, so there are two jumps $2$. Put $n=r-1$ and choose $B_n$ as in Lemma \ref{lem:asym-poly}(b). Set
\[
 f(x)=x^2+\frac1{x^2}+\frac1x+\frac1{B_n(x)}.
\]
The two jump-$2$ points are $0$ and $\infty$; all other poles are simple. A descended projective transformation preserving $\{0,\infty\}$ is either $x\mapsto\lambda x$ or $x\mapsto\lambda/x$. In the first case the simple branch set and Lemma \ref{lem:asym-poly}(b) force $\lambda=1$. In the second case, the left side of \eqref{eq:lift} has principal part $(\lambda^2-u)x^{-2}-ux^{-1}$ at $0$: the transformed terms $x^2/\lambda^2$ and $x/\lambda$ are polynomials, $1/B_n(\lambda/x)$ vanishes at $0$ because $\deg B_n\ge1$, and $u/B_n(x)$ is regular at $0$ because $B_n(0)\ne0$. Since $u\ne0$, this is a pole of order $1$ or $2$, impossible for an Artin--Schreier coboundary. Hence $M=1$.

\smallskip
\noindent\emph{Case 4: $m\ge3$.}
If $r\ge3$, let $R(x)$ be a squarefree polynomial in $\Fp[x]$ of degree $r-2$ satisfying $R(0)R(1)\ne0$; for $r=2$ omit the term $1/R(x)$ below. Such an $R$ exists: in degree $1$ choose $x-\lambda$ with $\lambda\in\Fp\setminus\{0,1\}$, and in degree at least $2$ take a monic irreducible polynomial.

If $m=3$, set
\[
 f_0(x)=x^2+\frac1{x^2}+\frac1{(x-1)^2},
\]
and put $f=f_0$ for $r=2$ and $f=f_0+1/R(x)$ for $r\ge3$. The points $0,1,\infty$ have jump $2$ and all others jump $1$. The point action on $\{0,1,\infty\}$ lies in $S_3$. For the five nonidentity substitutions
\[
 1-x,\quad x^{-1},\quad (1-x)^{-1},\quad \frac{x}{x-1},\quad \frac{x-1}{x},
\]
comparison of the order-$2$ principal parts gives $u=1$, and a direct calculation yields
\[
\begin{array}{c|c}
M & f_0(M(x))-f_0(x)\\ \hline
1-x & 1-2x\\[1mm]
1/x & 1+\dfrac{2}{x-1}\\[2mm]
1/(1-x) & -2x+2-\dfrac{2}{x}\\[2mm]
x/(x-1) & -2x+3+\dfrac{2}{x-1}-\dfrac{2}{x}\\[2mm]
(x-1)/x & 2+\dfrac{2}{x-1}-\dfrac{2}{x}.
\end{array}
\]
Each row contains a nonzero simple pole or a nonconstant polynomial term of degree $1$, neither of which can occur in $h^p-h$. When a row has a simple pole at $0$ or $1$, both $1/R(x)$ and $1/R(M(x))$ are regular there: the transformation $M$ permutes $0,1,\infty$, we have $R(0)R(1)\ne0$, and $1/R$ is regular at $\infty$. For $M(x)=1-x$, the obstruction is instead the nonconstant polynomial $1-2x$; the additional term $1/R(1-x)-1/R(x)$ is a proper rational function, hence regular at $\infty$, and cannot cancel its degree-one polynomial part. Thus the term involving $R$ cannot remove the displayed obstruction, and $M=1$.

If $(p,m)=(3,4)$, choose a squarefree polynomial $S_r\in\mathbb F_3[x]$ of degree $r-1$, with $S_r(0)\ne0$, whose root set is not stable under $x\mapsto-x$: take $S_2=x-1$, $S_3=x^2+x+2$, and for $r\ge4$ take $S_r=(x-1)Q_{r-2}(x)$ with $Q_{r-2}$ monic irreducible of degree $r-2\ge2$. Set
\[
 f(x)=x^4+\frac1{x^2}+\frac1{S_r(x)}.
\]
The jumps $4$ at $\infty$ and $2$ at $0$ force every descended transformation to be $x\mapsto\lambda x$. Principal-part comparison gives $u=\lambda^{-2}=\lambda^4$, hence $\lambda^6=1$, i.e.\ $(\lambda^2-1)^3=0$, and $\lambda=\pm1$; the chosen simple branch set excludes $\lambda=-1$.

In all remaining cases ($m\ge4$ if $p\ge5$, and $m\ge5$ if $p=3$) choose $a,b,c$ as in Lemma \ref{lem:triple}. If they are pairwise distinct, put
\[
 f(x)=x^c+\frac1{x^a}+\frac1{(x-1)^b}
\]
for $r=2$ and add $1/R(x)$ for $r\ge3$. The three distinct jump labels, all at least $2$, force $0,1,\infty$ to be fixed pointwise, so $M=1$. If exactly two jumps coincide, relabel so that $a=b=d$, $c=e\ne d$, and put
\[
 f(x)=x^e+\frac1{x^d}+\frac1{(x-1)^d}+\frac1{x-1},
\]
again adding $1/R(x)$ for $r\ge3$. The only possible nontrivial base substitution is $x\mapsto1-x$. Leading order-$d$ terms at $0$ give $u=(-1)^d$, and after cancellation the transformed term $1/((1-x)-1)=-1/x$ produces the nonzero simple pole $-1/x$ at $0$; all other terms are regular at $0$. This cannot be an Artin--Schreier coboundary. Thus $M=1$.

In every case every geometric automorphism induces the identity on the base and therefore belongs to the deck group $C_p$.
\end{proof}

\section{Proofs of the main theorems and consequences}\label{sec:proofs}

\begin{proof}[Proof of Theorem \ref{thm:main}]
If $r=0$, the assertion is exactly Theorem~\ref{thm:zero-rank}. Suppose $r\ge1$. If $r=1$, Proposition~\ref{prop:two-poles} gives existence for every $N\ge3$, while Proposition~\ref{prop:N2-impossible} excludes the boundary pair $(1,2)$. If $r\ge2$, Proposition~\ref{prop:many-poles} gives existence for every $N\ge2r$.

It remains to record necessity of the inequality in positive $p$-rank. If $\gamma(X)=r(p-1)$, Proposition~\ref{prop:formulas} gives exactly $r+1$ poles, with reduced positive lower jumps $d_1,\ldots,d_{r+1}$. Hence
\[
 N=-2+\sum_{i=1}^{r+1}(d_i+1)
   =\sum_{i=1}^{r+1}d_i+r-1
   \ge2r.
\]
This proves the classification.
\end{proof}

\begin{proof}[Proof of Corollary \ref{thm:poonen-type}]
Take $r=1$ in Theorem~\ref{thm:main}. The condition in part~\textup{(ii)} is then exactly $N\ge3$.
\end{proof}

\begin{corollary}\label{cor:finite-fields}
Let $q$ be any power of an odd prime $p$. For every pair $(r,N)$ allowed by Theorem \ref{thm:main}, there exists a curve $X/\mathbb F_q$ with the same genus and $p$-rank and with
\[
 \Aut_{\overline{\mathbb F}_q}(X)\cong C_p.
\]
\end{corollary}

\begin{proof}
Base-change the curve constructed over $\Fp$ to $\mathbb F_q$. The geometric curve is unchanged.
\end{proof}

\begin{corollary}\label{cor:jacobian}
Let $X$ be one of the curves in Theorem \ref{thm:main}, and let $(J_X,\Theta_X)$ be its principally polarized Jacobian; here $\Aut(J_X,\Theta_X)$ denotes the group of origin-preserving automorphisms of the principally polarized abelian variety. Then $X$ is nonhyperelliptic and
\[
 \Aut_{\Fpb}(J_X,\Theta_X)\cong C_p\times C_2\cong C_{2p}.
\]
\end{corollary}

\begin{proof}
In every case of Theorem~\ref{thm:main} one has $N\ge3$, so $g(X)=\frac{p-1}2N\ge N\ge3$. Since $p$ is odd and $\Aut_{\Fpb}(X)\cong C_p$, the curve has no involution and hence is not hyperelliptic. The natural action on divisor classes gives an injective homomorphism
\[
 \Aut_{\Fpb}(X)\hookrightarrow \Aut_{\Fpb}(J_X,\Theta_X).
\]
By the strong Torelli theorem \cite[Theorem~12.1]{Milne}, every automorphism of the principally polarized Jacobian is induced, up to sign, by an automorphism of the curve: every polarized group automorphism is of the form $\varepsilon\,\alpha_*$ with $\varepsilon\in\{\pm1\}$ and $\alpha\in\Aut_{\Fpb}(X)$, and in the nonhyperelliptic case the pair $(\varepsilon,\alpha)$ is unique. Thus
\[
 \Aut_{\Fpb}(J_X,\Theta_X)=\Aut_{\Fpb}(X)\,\langle[-1]\rangle.
\]
The involution $[-1]$ is central, and by the uniqueness statement it is not induced by an automorphism of $X$. Consequently the product is direct:
\[
 \Aut_{\Fpb}(J_X,\Theta_X)\cong \Aut_{\Fpb}(X)\times\{\pm1\}
 \cong C_p\times C_2\cong C_{2p}.
\]
\end{proof}

\section{Summary of the models and further remarks}\label{sec:remarks}

Table~\ref{tab:models} collects the models; in all of them $f\in\Fp(x)$. The auxiliary polynomials appearing in several rows are chosen in the preceding lemmas and propositions.

\begin{table}[ht]
\centering
\renewcommand{\arraystretch}{1.25}
\begin{tabular}{@{}lll@{}}
\toprule
Range & $f(x)$ & Reference\\
\midrule
$r=0$, $3\le N\le6$ & see the table in Lemma~\ref{lem:small-one-pole} & Lemma~\ref{lem:small-one-pole}\\
$r=0$, $N\ge7$ & $x^d+x^{d-s}+x^{d-t}$ or $x^d+x^{d-2}+x^{d-5}$, $d=N+1$ & Prop.~\ref{prop:one-pole-large}\\
$r=1$, $N\ge3$ & $x^b+x^t+x^{-a}$;\quad $(p,N)=(3,4)$: $x^2+x+x^{-2}$ & Prop.~\ref{prop:two-poles}\\
$r\ge2$, $m=0$, $r\ne p-1$ & $1/C_{r+1}(x)$ & Prop.~\ref{prop:many-poles}\\
$r=p-1\ge4$, $m=0$ & $1/(x^p+x^2+x+1)$ & Lemma~\ref{lem:p-point-set}\\
$p=3$, $r=2$, $m=0$ & $x^2/(x^3+2x^2+1)$ & Lemma~\ref{lem:p3-boundary}\\
$r\ge2$, $m=1$ & $x^2+x+\frac1x+\frac1{x-1}$\ ($r=2$);\quad $x^2+\frac1{A_r}$ & Prop.~\ref{prop:many-poles}\\
$r\ge2$, $m=2$, $p\ge5$ & $x^3+\frac1x+\frac1{x-1}$\ ($r=2$);\quad $x^3+\frac1{A_r}$ & Prop.~\ref{prop:many-poles}\\
$r\ge2$, $m=2$, $p=3$ & $x^2+x^{-2}+x^{-1}+1/B_{r-1}$ & Prop.~\ref{prop:many-poles}\\
$r\ge2$, $m=3$ & $x^2+x^{-2}+(x-1)^{-2}\ (+1/R)$ & Prop.~\ref{prop:many-poles}\\
$p=3$, $r\ge2$, $m=4$ & $x^4+x^{-2}+1/S_r$ & Prop.~\ref{prop:many-poles}\\
remaining $m$ & $x^c+x^{-a}+(x-1)^{-b}\,(+(x-1)^{-1})\,(+1/R)$ & Prop.~\ref{prop:many-poles}\\
\bottomrule
\end{tabular}
\caption{The explicit models realizing Theorem~\ref{thm:main}; here $m=N-2r$.}
\label{tab:models}
\end{table}

\begin{remark}
Theorem~\ref{thm:main} leaves no numerical existence cases open inside the Artin--Schreier locus. For $r\ge1$, Proposition~\ref{prop:formulas} forces $N\ge2r$, every boundary value with $r\ge2$ is attained, and $(r,N)=(1,2)$ is genuinely impossible. For $r=0$, the unique lower jump is $N+1$, so $p\nmid N+1$ is necessary; Proposition~\ref{prop:zero-obstructions} shows that $N=1$, $N=2$, and $(p,N)=(3,3)$ carry unavoidable extra automorphisms, while Lemma~\ref{lem:small-one-pole} and Proposition~\ref{prop:one-pole-large} realize every remaining admissible value.
\end{remark}

\begin{remark}
The Valentini--Madan theorem is essential at the sharp boundary. A Castelnuovo--Severi argument alone requires $g>(p-1)^2$ (Remark~\ref{rem:CS}). The only places where the numerical invariants of our models coincide with those of an exceptional family are the ordinary boundary $(r,N)=(p-1,2p-2)$ and the one-pole genera with $N+1\mid p+1$; these are handled by Lemmas~\ref{lem:E-rigid} and~\ref{lem:monomial}, respectively.
\end{remark}

\begin{remark}
For $p=2$, Artin--Schreier curves are hyperelliptic and the deck group is the hyperelliptic involution. Zhu proved that over $\mathbb F_2$ one can prescribe the $2$-rank and obtain no extra automorphisms in every relevant genus \cite{Zhu}. Theorem \ref{thm:main} is an odd-characteristic counterpart for minimal automorphism groups inside the Artin--Schreier locus. It does not answer the prime-field questions recorded by Achter--Glass--Pries \cite[Remark~1.2]{AGP} for arbitrary curves and for hyperelliptic curves in odd characteristic: those questions ask for trivial automorphism group, respectively only the hyperelliptic involution, whereas the deck group $C_p$ is intrinsic to the Artin--Schreier locus.
\end{remark}

We end with two natural questions suggested by the method.

\begin{question}
Let $G$ be a finite $p$-group. For which genera and $p$-ranks does there exist a $G$-Galois cover $X\to\PP$ defined over $\Fp$ with $\Aut_{\Fpb}(X)=G$? For $G=C_{p^n}$ (Artin--Schreier--Witt covers) the descent step would require an analogue of the Valentini--Madan theorem.
\end{question}

\begin{question}
Within a fixed $p$-rank stratum, can one also prescribe the Newton polygon, or the $a$-number, of an Artin--Schreier curve $X/\Fp$ with $\Aut_{\Fpb}(X)\cong C_p$? For instance, which genera admit a supersingular such curve?
\end{question}

\section*{Acknowledgements}
Saeed Tafazolian was partially supported by CNPq grant no.~302774/2025-4 and FAPESP grant no.~2024/00923-6.

\end{document}